\documentclass[12pt]{article}
\usepackage{amsmath,amsfonts,amsthm,amssymb}
\usepackage{graphicx}
\usepackage{fancyhdr}
\usepackage{enumerate}
\usepackage[justification=centering]{caption}
\usepackage[labelformat=simple]{subcaption}

\providecommand{\abs}[1]{\left\lvert#1\right\rvert}

\providecommand{\zee}{\mathbb{Z}}

\providecommand{\be}{\begin{equation}}
\providecommand{\ee}{\end{equation}}
\def\ba#1\ea{\begin{align}#1\end{align}}
\def\bas#1\eas{\begin{align*}#1\end{align*}}

\providecommand{\gamh}[1]{\ensuremath{\Gamma_{#1}(H)}}
\providecommand{\gamhg}[1]{\ensuremath{\Gamma_{\geq #1}(H)}}

\providecommand{\rad}[1]{\ensuremath{\rho(#1)}}
\renewcommand{\deg}[1]{\ensuremath{d(#1)}}
\DeclareMathOperator{\dist}{dist}

\newtheorem{theorem}{Theorem}[section]
\theoremstyle{plain}

\newtheorem{corollary}[theorem]{Corollary}

\newtheorem{lemma}[theorem]{Lemma}

\newtheorem{problem}[theorem]{Problem}
\newtheorem{proposition}[theorem]{Proposition}
\newtheorem{remark}[theorem]{Remark}

\begin{document}

\title{Maximizing the order of a regular graph of given valency and second eigenvalue}
\author{Sebastian M. Cioab\u{a}$^{*1}$, Jack H. Koolen$^{\dagger2}$\\
Hiroshi Nozaki$^\ddagger$ \& Jason R. Vermette$^{*1}$\vspace{10pt}\\
\small $^*$Department of Mathematical Sciences,\vspace{-3pt}\\
\small University of Delaware, Newark DE 19716-2553, USA\vspace{4pt}\\
\small $^\dagger$ School of Mathematical Sciences,\vspace{-3pt}\\
\small  University of Science and Technology of China,\vspace{-3pt}\\
\small Wen-Tsun Wu Key Laboratory of the Chinese Academy of Sciences, Hefei, Anhui, China \vspace{4pt}\\
\small $^\ddagger$  Department of Mathematics,\vspace{-3pt}\\
\small   Aichi University of Education, 1 Hirosawa, Igaya-cho, Kariya, Aichi 448-8542, Japan \vspace{4pt}\\
\small {\tt cioaba@udel.edu, koolen@ustc.edu.cn}\vspace{-3pt}\\
\small {\tt hnozaki@auecc.aichi-edu.ac.jp, vermette@udel.edu}}
\date{\today}
\maketitle

\footnotetext[1]{
Research partially supported by the National Security Agency grant
H98230-13-1-0267.}
\footnotetext[2]{JHK is partially supported by the National Natural Science Foundation of China (No.\ 11471009). He also acknowledges the financial support of the Chinese Academy of Sciences under its `100 talent' program.}
\footnotetext[3]{HN is partially supported by JSPS Grants-in-Aid for Scientific Research No.\ 25800011.
}

\begin{abstract}
From Alon and Boppana, and Serre, we know that for any given integer $k\geq 3$ and real number $\lambda<2\sqrt{k-1}$, there are only finitely many $k$-regular graphs whose second largest eigenvalue is at most $\lambda$. In this paper, we investigate the largest number of vertices of such graphs. 
\end{abstract}

\section{Introduction}

For a $k$-regular graph $G$ on $n$ vertices, we denote by $\lambda_1(G)=k>\lambda_2(G)\geq\ldots\geq\lambda_n(G)=\lambda_{\min}(G)$ the eigenvalues of the adjacency matrix of $G$. For a general reference on the eigenvalues of graphs, see \cite{BH,GRb}.

The second eigenvalue of a regular graph is a parameter of interest in the study of graph connectivity and expanders (see \cite{A,BH,HLW} for example). In this paper, we investigate the maximum order $v(k,\lambda)$ of a connected $k$-regular graph whose second largest eigenvalue is at most some given parameter $\lambda$.  As a consequence of work of Alon and Boppana, and of Serre \cite{A,C,F,HLW,KS,LPS,Mo,Ni2,Ni1,Serre1}, we know that $v(k,\lambda)$ is finite for $\lambda<2\sqrt{k-1}$. The recent result of Marcus, Spielman and Srivastava \cite{MSS} showing the existence of infinite families of Ramanujan graphs of any degree at least $3$ implies that $v(k,\lambda)$ is infinite for $\lambda\geq2\sqrt{k-1}$. 

For any $\lambda<0$, the parameter $v(k,\lambda)$ can be determined using the fact that a graph with only one nonnegative eigenvalue is a  complete graph. Indeed, if a graph has only one nonnegative eigenvalue, then it must be connected. If a connected graph $G$ is not a complete graph, then $G$ contains an induced subgraph isomorphic to $K_{1,2}$, so Cauchy eigenvalue interlacing (see \cite[Proposition 3.2.1]{BH}) implies $\lambda_2(G)\geq\lambda_2(K_{1,2})=0$, contradiction. Thus $v(k,\lambda)=k+1$ for any $\lambda<0$ and the unique graph meeting this bound is $K_{k+1}$. The parameter $v(k,0)$ can be determined using the fact that a graph with exactly one positive eigenvalue must be a complete multipartite graph (see \cite[page 89]{BCN}). The largest $k$-regular complete multipartite graph is the complete bipartite graph $K_{k,k}$, since a $k$-regular $t$-partite graph has $t k/(t-1)$ vertices. Thus $v(k,0)=2k$, and $K_{k,k}$ is the unique graph meeting this bound. The values of $v(k,-1)$ and $v(k,0)$ also follow from Theorem \ref{bound} in Section \ref{lp} below.

Results from Bussemaker, Cvetkovi\'c and Seidel  \cite{BCS} and Cameron, Goethals, Seidel, and Shult \cite{CGSS} give a characterization of the regular graphs with smallest eigenvalue $\lambda_{\min}\geq-2$. Since the second eigenvalue of the complement of a regular graph is $\lambda_2=-1-\lambda_{\min}$, the regular graphs with second eigenvalue $\lambda_2\leq1$ are also characterized. This characterization can be used to find $v(k,1)$ (see Section \ref{vk1}).
 
The values remaining to be investigated are $v(k,\lambda)$ for $1<\lambda<2\sqrt{k-1}$.  The parameter $v(k,\lambda)$ has been studied by Teranishi and Yasuno \cite{TeYa} and H{\o}holdt and Justesen \cite{HJ} for the class of bipartite graphs in connection with problems in design theory, finite geometry and coding theory. Some results involving $v(k,\lambda)$ were obtained by Koledin and Stan\'{i}c \cite{KoSt1,KoSt2, Stanic} and Richey, Shutty and Stover \cite{RSS} who implemented Serre's quantitative version of the Alon--Boppana Theorem \cite{Serre1} to obtain upper bounds for $v(k,\lambda)$ for several values of $k$ and $\lambda$. For certain values of $k$ and $\lambda$, Richey, Shutty and Stover \cite{RSS} made some conjectures about $v(k,\lambda)$. We will prove some of their conjectures and disprove others in this paper. Reingold, Vadhan and Wigderson \cite{RVW} used regular graphs with small second eigenvalue as the starting point of their iterative construction of infinite families of expander using the zig-zag product. Guo, Mohar, and Tayfeh-Rezaie \cite{GM,Mo2,MT} studied a similar problem involving the median eigenvalue. Nozaki \cite{N} investigated a related, but different problem from the one studied in our paper, namely finding the regular graphs of given valency and order with smallest second eigenvalue. Amit, Hoory and Linial \cite{AHL} studied a related problems of minimizing $\max(|\lambda_2|,|\lambda_n|)$ for regular graphs of given order $n$, valency $k$ and girth $g$.

In this paper, we determine $v(k,\lambda)$ explicitly for several values of $(k,\lambda)$, confirming or disproving several conjectures in \cite{RSS}, 
\begin{table}\small%
\centering
\caption{Summary of our Results for $k\leq22$}
\label{summarytable}
\begin{tabular}{|c|c|c|c|c|c|c|c|} 
\hline
$(k,\lambda)$ & $v(k,\lambda)$&& $(k,\lambda)$ & $v(k,\lambda)$&&$(k,\lambda)$ & $v(k,\lambda)$\\
\hline
 $(2,-1)$ & 3 &   & $(7,1)$ & 18 &   & $\left(14,\sqrt{13}\right)$ & 366 \\
 $(2,0)$ & 4 &   & $(7,2)$ & 50 &   & $\left(14,\sqrt{26}\right)$ & 4760 \\
 $\left(2,\frac{1}{2} \left(\sqrt{5}-1\right)\right)$ & 5 &   & $(8,-1)$ & 9 &   &
   $\left(14,\sqrt{39}\right)$ & 804468 \\
 $(2,1)$ & 6 &   & $(8,0)$ & 16 &   & $(15,-1)$ & 16 \\
 $\left(2,\sqrt{2}\right)$ & 8 &   & $(8,1)$ & 21 &   & $(15,0)$ & 30 \\
 $\left(2,\frac{1}{2} \left(\sqrt{5}+1\right)\right)$ & 10 &   & $\left(8,\sqrt{7}\right)$ &
   114 &   & $(15,1)$ & 32 \\
 $\left(2,\sqrt{3}\right)$ & 12 &   & $\left(8,\sqrt{14}\right)$ & 800 &   & $(16,-1)$ & 17
   \\
 $(3,-1)$ & 4 &   & $\left(8,\sqrt{21}\right)$ & 39216 &   & $(16,0)$ & 32 \\
 $(3,0)$ & 6 &   & $(9,-1)$ & 10 &   & $(16,1)$ & 34 \\
 $(3,1)$ & 10 &   & $(9,0)$ & 18 &   & $(16,2)$ & 77 \\
 $\left(3,\sqrt{2}\right)$ & 14 &   & $(9,1)$ & 24 &   & $(17,-1)$ & 18 \\
 $\left(3,\sqrt{3}\right)$ & 18 &   & $\left(9,2 \sqrt{2}\right)$ & 146 &   & $(17,0)$ & 34
   \\
 $(3,2)$ & 30 &   & $(9,4)$ & 1170 &   & $(17,1)$ & 36 \\
 $\left(3,\sqrt{6}\right)$ & 126 &   & $\left(9,2 \sqrt{6}\right)$ & 74898 &   & $(18,-1)$ &
   19 \\
 $(4,-1)$ & 5 &   & $(10,-1)$ & 11 &   & $(18,0)$ & 36 \\
 $(4,0)$ & 8 &   & $(10,0)$ & 20 &   & $(18,1)$ & 38 \\
 $(4,1)$ & 9 &   & $(10,1)$ & 27 &   & $\left(18,\sqrt{17}\right)$ & 614 \\
 $\left(4,\sqrt{5}-1\right)$ & 10 &   & $(10,2)$ & 56 &   & $\left(18,\sqrt{34}\right)$ &
   10440 \\
 $\left(4,\sqrt{3}\right)$ & 26 &   & $(10,3)$ & 182 &   & $\left(18,\sqrt{51}\right)$ &
   3017196 \\
 $(4,2)$ & 35 &   & $\left(10,3 \sqrt{2}\right)$ & 1640 &   & $(19,-1)$ & 20 \\
 $\left(4,\sqrt{6}\right)$ & 80 &   & $\left(10,3 \sqrt{3}\right)$ & 132860 &   & $(19,0)$ &
   38 \\
 $(4,3)$ & 728 &   & $(11,-1)$ & 12 &   & $(19,1)$ & 40 \\
 $(5,-1)$ & 6 &   & $(11,0)$ & 22 &   & $(20,-1)$ & 21 \\
 $(5,0)$ & 10 &   & $(11,1)$ & 24 &   & $(20,0)$ & 40 \\
 $(5,1)$ & 16 &   & $(12,-1)$ & 13 &   & $(20,1)$ & 42 \\
 $(5,2)$ & 42 &   & $(12,0)$ & 24 &   & $\left(20,\sqrt{19}\right)$ & 762 \\
 $\left(5,2 \sqrt{2}\right)$ & 170 &   & $(12,1)$ & 26 &   & $\left(20,\sqrt{38}\right)$ &
   14480 \\
 $\left(5,2 \sqrt{3}\right)$ & 2730 &   & $\left(12,\sqrt{11}\right)$ & 266 &   &
   $\left(20,\sqrt{57}\right)$ & 5227320 \\
 $(6,-1)$ & 7 &   & $\left(12,\sqrt{22}\right)$ & 2928 &   & $(21,-1)$ & 22 \\
 $(6,0)$ & 12 &   & $\left(12,\sqrt{33}\right)$ & 354312 &   & $(21,0)$ & 42 \\
 $(6,1)$ & 15 &   & $(13,-1)$ & 14 &   & $(21,1)$ & 44 \\
 $\left(6,\sqrt{5}\right)$ & 62 &   & $(13,0)$ & 26 &   & $(22,-1)$ & 23 \\
 $\left(6,\sqrt{10}\right)$ & 312 &   & $(13,1)$ & 28 &   & $(22,0)$ & 44 \\
 $\left(6,\sqrt{15}\right)$ & 7812 &   & $(14,-1)$ & 15 &   & $(22,1)$ & 46 \\
 $(7,-1)$ & 8 &   & $(14,0)$ & 28 &   & $(22,2)$ & 100 \\
 $(7,0)$ & 14 &   & $(14,1)$ & 30 &   &  &  \\
\hline
\end{tabular}
\end{table}
and we find the graphs (in many cases unique) which meet our bounds. In many cases these graphs are distance-regular. For definitions and notations related to distance-regular graphs, we refer the reader to \cite[Chapter 12]{BH}. Table \ref{summarytable} contains a summary of the values of $v(k,\lambda)$ that we found for $k\leq22$. Table \ref{tab:1} contains six infinite families of graphs and seven sporadic graphs meeting the bound $v(k,\lambda)$ for some values of $k,\lambda$ due to Theorem \ref{bound}. Table \ref{tab:2} illustrates that the graphs in Table \ref{tab:1} that meet the bound $v(k,\lambda)$ also meet the bound $v(k,\lambda')$ for certain $\lambda'>\lambda$ due to Proposition \ref{prop:large}.

\section{Linear programming method}\label{lp}

In this section, we give a bound for $v(k,\lambda)$ using the linear programming method developed by Nozaki $\cite{N}$. 
Let $F_i=F_i^{(k)}$ be orthogonal polynomials defined by the three-term recurrence relation:  
\begin{equation*}
F_0^{(k)}(x)=1, \qquad F_1^{(k)}(x)=x, \qquad  F_2^{(k)}(x)=x^2 - k,   
\end{equation*}
and
\begin{equation*}
F_i^{(k)}(x)=x F_{i-1}^{(k)}(x)- (k-1) F_{i-2}^{(k)}(x)
\end{equation*}
for $i\geq 3$. The following is called the linear programming bound for regular graphs. 
\begin{theorem}[Nozaki \cite{N}]
\label{thm:lp_bound}
Let $G$ be a connected $k$-regular graph with $v$ vertices. 
Let $\lambda_1=k, \lambda_2, \ldots, \lambda_n $ be the distinct eigenvalues of $G$. 
Suppose there exists a polynomial $f(x)=\sum_{i\geq 0}  f_i F_i^{(k)}(x)$ such that  
$f(k)>0$, $f(\lambda_i) \leq 0$ for any $i\geq 2$, $f_0>0$, and $f_i \geq 0$ for any $i\geq 1$. Then we have
\begin{equation*} 
v \leq \frac{f(k)}{f_0}. 
\end{equation*}
\end{theorem}
Using Theorem \ref{thm:lp_bound}, Nozaki \cite{N} proved Theorem \ref{thm:ext} below. Note that the paper \cite{N} deals only with the problem of minimizing the second eigenvalue of a regular graph of given order and valency. While related to the problem of estimating $v(k,\lambda)$, the problem considered by Nozaki in \cite{N} is quite different from the one we study in this paper.
\begin{theorem}[Nozaki \cite{N}] \label{thm:ext}
Let $G$ be a connected $k$-regular graph of girth $g$, with $v$ vertices. Assume the number of distinct eigenvalues of $G$ is $d+1$. If $g \geq 2d$ holds, then $G$ has the smallest second-largest eigenvalue in all $k$-regular graphs with $v$ vertices. 
\end{theorem}
Note also that while Table~\ref{tab:1} is similar to \cite[Table 2]{N}, the problems and tools in our paper are significantly different from the ones in \cite{N}.

Let $T(k,t,c)$ be the $t \times t$ tridiagonal matrix  with lower diagonal $(1,1,\ldots,1,c)$, upper diagonal $(k,k-1,\ldots, k-1)$, and with constant row sum $k$, where $c$ is a positive real number. Theorem~\ref{bound} is the main theorem in this section and gives a new comprehension of the linear programming method and a general upper bound for $v(k,\lambda)$ without any assumption regarding the existence of some particular graphs.
\begin{theorem} \label{bound}
If $\lambda_2$ is the second largest eigenvalue of $T(k,t,c)$, then
\begin{equation}\label{main}
v(k,\lambda_2) \leq M(k,t,c)=1+\sum_{i=0}^{t-3} k(k-1)^i+ \frac{k(k-1)^{t-2}}{c}. 
\end{equation}
Let $G$ be a $k$-regular connected graph with second largest eigenvalue at most $\lambda_2$, valency $k$, and $v(k, \lambda_2)$ vertices. Then $v(k, \lambda_2) = M(k, t, c)$ if and only if $G$ is distance-regular with quotient matrix $T(k, t, c)$ with respect to the distance-partition.
\end{theorem} 
\begin{proof}
We first show that the eigenvalues of $T$ that are not equal to $k$, coincide with the zeros of $\sum_{i=0}^{t-2}F_i(x)+F_{t-1}(x)/c$ (see also \cite[Section 4.1 B]{BCN}). 
Indeed, 
\[
[F_0,F_1,\ldots,F_{t-2},F_{t-1}/c]T=[xF_0,xF_1,\ldots,xF_{t-2}, (k-1)F_{t-2}+(k-c) F_{t-1}/c],
\]
and 
\begin{align*}
[F_0,F_1,\ldots,F_{t-2},F_{t-1}/c](T-xI) &=[0,0,\ldots,0, (k-1)F_{t-2}+(-x+k-c) F_{t-1}/c]\\
 &=[0,0,\ldots,0,(k-x)(\sum_{i=0}^{t-2}F_i+F_{t-1}/c)]\\
&=[0,0,\ldots,0,(k-x)((c-1)G_{t-2}+G_{t-1})/c]
\end{align*}
by the three-term recurrence relation, where $G_i(x)=\sum_{j=0}^iF_j(x)$. This equation implies that the zeros of $(k-x)((c-1)G_{t-2}+G_{t-1})$ are eigenvalues of $T$.
The monic polynomials $G_i$ form a sequence of orthogonal polynomials with respect to some positive weight on the interval $[-2\sqrt{k-1}, 2\sqrt{k-1}]$ \cite{N}. 
Since the zeros of $G_{t-2}$ and $G_{t-1}$ interlace on $[-2\sqrt{k-1}, 2\sqrt{k-1}]$, the zeros of $(k-x)((c-1)G_{t-2}+G_{t-1})$ are simple. Therefore all eigenvalues of $T$ coincide with the zeros of  $(k-x)((c-1)G_{t-2}+G_{t-1})$, and are simple. 
 
Let $\lambda_1=k>\lambda_2>\ldots>\lambda_{t}$ be the eigenvalues of $T$.  
We prove that the polynomial 
\begin{equation}
f(x)=\frac{1}{c}\cdot (x-\lambda_2)\prod_{i=3}^{t}(x-\lambda_i)^2=\sum_{i=0}^{2t-3}f_i F_i(x)
\end{equation}
satisfies $f_i>0$ for $i=0,1,\ldots, 2t-3$.  
Note that it trivially holds that $f(k)>0$, and $f(\lambda) \leq 0$ for any $\lambda \leq \lambda_2$. 
The polynomial $f(x)$ can be expressed as
\begin{equation}
f(x)=\frac{(c-1)G_{t-2}+G_{t-1}}{x-\lambda_2}\cdot \left(\sum_{i=0}^{t-2}F_i+F_{t-1}/c\right).
\end{equation}
By \cite[Proposition~3.2]{CK07},  $g(x)=((c-1)G_{t-2}+G_{t-1})/(x-\lambda_2)$ has  positive coefficients in terms of $G_{0},G_1,\ldots, G_{t-2}$. This implies that $g(x)$ has positive 
coefficients in terms of $F_{0},F_1,\ldots, F_{t-2}$. Therefore $f_i>0$ for $i=0,1,\ldots, 2t-3$ by \cite[Theorem~3]{N}. 

The polynomial $g(x)$ can be expressed as $g(x)=\sum_{i=0}^{t-2}g_i F_i(x)$.  
By \cite[Theorem~3]{N}, we get that $f_0=\sum_{i=0}^{t-2} g_i F_i(k)=g(k)$. Using Theorem \ref{thm:lp_bound} for $f(x)$, we obtain that
\begin{align*}
v(k,\lambda_2) &\leq \frac{f(k)}{f_0}=\sum_{i=0}^{t-2}F_i(k)+F_{t-1}(k)/c \\
&=1+\sum_{i=0}^{t-3} k(k-1)^i+ \frac{k(k-1)^{t-2}}{c}. 
\end{align*}
By \cite[Remark~2]{N}, the graph attaining the bound has girth at least $2t-2$, and at most $t$ distinct eigenvalues. Therefore the graph is a distance-regular graph with quotient matrix $T(k,t,c)$ by \cite[Theorem~6]{N} and \cite{DG81}. Conversely the distance-regular graph with  quotient matrix $T(k,t,c)$ clearly attains the bound $M(k,t,c)$. 
\end{proof}
\begin{remark}
The distance-regular graphs which have $T(k,t,c)$ as a quotient matrix of the distance partition are precisely the distance-regular graphs with intersection array $\{k,k-1,\ldots,k-1;1,\ldots,1,c\}$.
\end{remark}
\begin{corollary}
Let $H$ be a connected $k$-regular graph  with at least $M(k,t,c)$ 
vertices. Let $\lambda_2$ be the second largest eigenvalue of $T(k,t,c)$.
Then $\lambda_2 \leq \lambda_2(H)$ holds with equality only if $H$
meets the bound $M(k,t,c)$. 
\end{corollary} 

\begin{proof}
By Theorem~\ref{bound}, if $\lambda_2 > \lambda_2(H)$, then
the order of $H$ is at most $M(k,t,c)$. 
If the order of $H$ is equal to $M(k,t,c)$, then $H$ has  at most 
$t-1$ distinct eigenvalues by \cite[Remark~2]{N}.
However then  the order of $H$ is less than $M(k,t-1,1)$ by the Moore bound, a contradiction.  Therefore if $\lambda_2 > \lambda_2(H)$, then
the order of $H$ is less than $M(k,t,c)$. 
Namely if the order of $H$ is at least $M(k,t,c)$, then $\lambda_2 \leq \lambda_2(H)$. 
If $\lambda_2 = \lambda_2(H)$ holds, then the order of $H$ is bounded above by $M(k,t,c)$ in Theorem~\ref{bound}, and attains the bound. 
\end{proof}

We will discuss a possible second eigenvalue $\lambda_2$ of $T(k,t,c)$. 
Indeed for any $-1\leq \lambda<2\sqrt{k-1}$ there exist $t,c$ such that $\lambda$ is 
the second eigenvalue of $T(k,t,c)$. 
Let $\lambda^{(t)}$, $\mu^{(t)}$ be the largest zero of $G_t$, $F_t$, respectively. 
The zero $\lambda^{(t)}$ can be expressed by $\lambda^{(t)}=2\sqrt{k-1} \cos \theta$, where $\pi/(t+1) < \theta < \pi/t$ \cite[Section III.3]{BIb}. 
\begin{proposition}
The following hold:
 \begin{enumerate}
\item $\lambda^{(t)}< \mu^{(t)}$ for any $k$, $t$. 
\item $\mu^{(t-1)}< \lambda^{(t)}$ for $k \geq 5$ and any $t$, $k = 4$ and $t\leq 5$, or $k = 3$ and $t\leq 3$. 
\item $\mu^{(t-1)}> \lambda^{(t)}$ for $k = 4$ and $t\geq 6$, or $k = 3$ and $t\geq 4$. 
\end{enumerate}
\end{proposition}
\begin{proof} 
Since $F_t(\lambda^{(t)})=G_t(\lambda^{(t)})-G_{t-1}(\lambda^{(t)})=-G_{t-1}(\lambda^{(t)})<0$, 
we have $\lambda^{(t)}< \mu^{(t)}$ for any $k$, $t$. Note that $F_t$ has a unique zero greater than $\lambda^{(t)}$. By the equality $(k-1)F_{t-1}=(k-1-x)G_{t-1}+G_t$, we obtain that
\begin{align*}
(k-1)F_{t-1}(\lambda^{(t)})&=(k-1-\lambda^{(t)})G_{t-1}(\lambda^{(t)})+G_t(\lambda^{(t)})\\
&=(k-1-\lambda^{(t)})G_{t-1}(\lambda^{(t)})\\
&=(k-1-2\sqrt{k-1} \cos  \theta )G_{t-1}(\lambda^{(t)})\\
&\begin{cases}
> (k-1-2\sqrt{k-1}\cos \frac{\pi}{t+1})G_{t-1}(\lambda^{(t)}) \geq 0 \text{ for $(k,t)$ in (2),}\\
< (k-1-2 \sqrt{k-1}\cos \frac{\pi}{t}) G_{t-1}(\lambda^{(t)}) \leq 0 \text{ for $(k,t)$ in (3).}
\end{cases}
\end{align*}
This finishes the proof of the proposition. 
\end{proof}
\begin{remark} \label{rem:zero}
The second largest eigenvalue $\lambda_2(c)$ of $T(k,t,c)$ is the largest zero  of $(c-1)G_{t-2}+G_{t-1}$. 
Since the zeros of $G_{t-2}$ and $G_{t-1}$ interlace, 
$\lambda_2(c)$ is a monotonically decreasing function in $c$. In particular, 
$\lim_{c\rightarrow \infty }\lambda_2(c)=\lambda^{(t-2)}$, $\lambda_2(1)=\lambda^{(t-1)}$, 
and $\lim_{c \rightarrow 0 }\lambda_2(c)=\mu^{(t-1)}$. 
\end{remark}
Note that both $F_i$ and $G_i$ form a sequence of orthogonal polynomials with
respect to some positive weight on the interval $[-2\sqrt{k-1}, 2\sqrt{k-1}]$. 
By Remark \ref{rem:zero}, the second eigenvalue $\lambda_2(t,c)$ of $T(k,t,c)$ may equal all possible values between $\lambda_2(2,1)=-1$ and $\lim_{t\rightarrow \infty} \lambda_2(t,c)=2\sqrt{k-1}$. The following proposition shows that we may assume $c\geq 1$ in Theorem~\ref{bound} to obtain better bounds. 
\begin{proposition} \label{prop:c}
For any $\lambda$ such that $\lambda^{(t-1)} < \lambda < \mu^{(t-1)}$, there exist $0< c_1 < 1$, $c_2 > 0$ such that 
both the second-largest eigenvalues of $T(k,t,c_1)$ and $T(k,t+1,c_2)$ are $\lambda$. Then we have
$M(k,t,c_1)>M(k,t+1,c_2)$. 
\end{proposition}
\begin{proof}
Because $(c_1-1)G_{t-2}(\lambda)+G_{t-1}(\lambda)=0$, we get $c_1=-\frac{G_{t-1}(\lambda)-G_{t-2}(\lambda)}{G_{t-2}(\lambda)}=-F_{t-1}(\lambda)/G_{t-2}(\lambda)$. 
Similarly $c_2=-F_{t}(\lambda)/G_{t-1}(\lambda)$. Note that $F_{t-1}(\lambda)=-c_1 G_{t-2}(\lambda)<0$ and $F_{t}(\lambda)=-c_2G_{t-1}(\lambda)<0$. 
Therefore
\begin{align*}
M(k,t,c_1)-M(k,t+1,c_2) 
& =k(k-1)^{t-2}\big(\frac{1}{c_1}-1-\frac{1}{c_2}(k-1)\big)\\
&= k(k-1)^{t-2}\big(-\frac{G_{t-2}(\lambda)}{F_{t-1}(\lambda)}-1+(k-1)\frac{G_{t-1}(\lambda)}{F_{t}(\lambda)}\big)\\
&= k(k-1)^{t-2}\big(-\frac{G_{t-1}(\lambda)}{F_{t-1}(\lambda)}+(k-1)\frac{G_{t-1}(\lambda)}{F_{t}(\lambda)}\big)\\
&=\frac{k(k-1)^{t-2}G_{t-1}(\lambda)}{F_{t-1}(\lambda)F_{t}(\lambda)}\big(-F_{t}(\lambda)+(k-1)F_{t-1}(\lambda) \big)\\
&=\frac{k(k-1)^{t-2}(k-\lambda)G_{t-1}(\lambda)^2}{F_{t-1}(\lambda)F_{t}(\lambda)}>0. \qquad \qedhere
\end{align*}
\end{proof}
Table \ref{tab:1} shows the known examples attaining the bound $M(k,t,c)$. The incidence graphs of $PG(2,q)$, $GQ(q,q)$, and $GH(q,q)$ are known to be unique for $q\leq8$, $q\leq4$, and $q\leq2$, respectively (see, for example, \cite[Table 6.5 and the following comments]{BCN}). The incidence graphs of $PG(2,2)$, $GQ(2,2)$, and $GH(2,2)$ are the Heawood graph, the Tutte-Coxeter graph (or Tutte 8-cage), and the Tutte 12-cage, respectively.

\begin{table}[ht!]
\caption{Known graphs meeting the bound $M(k,t,c)$}\label{tab:1}
\begin{tabular}{|c|c|c|c|c|} 
\hline
$(k,\lambda)$ & $v(k,\lambda)$& Graph meeting bound& Unique?& Ref. \\
\hline 
$(2,2\cos(2\pi/n))$ & $n$ & $n$-cycle $C_n$& yes & \\
$(k,-1)$ &$k+1$ & Complete graph $K_{k+1}$ & yes & \\ 
$(k,0)$ & $2k$ & Complete bipartite graph $K_{k,k}$ & yes & \\
$(q+1,\sqrt{q})$ &$2(q^2+q+1) $& incidence graph of $PG(2,q)$ & ? & \cite{BCN,S66}\\
$(q+1,\sqrt{2q})$&$2(q+1)(q^2+1)$     &incidence graph of $GQ(q,q)$ & ? & \cite{B66,BCN}\\
$(q+1,\sqrt{3q})$ & $2(q+1)(q^4+q^2+1)$ &incidence graph of $GH(q,q)$& ? &\cite{B66,BCN} \\
$(3,1)$& $10$ & Petersen graph& yes & \cite{HS}\\
$(4,2)$& $35$ & Odd graph $O_4$& yes & \cite{M}\\
$(7,2)$& $50$ & Hoffman--Singleton graph& yes & \cite{HS}\\
$(5,1)$& $16$ & Clebsch graph & yes & \cite{GRb,S68}\\
$(10,2)$& $56$ & Gewirtz graph & yes &\cite{BH93,G69}\\
$(16,2)$& $77$ & $M_{22}$ graph & yes & \cite{B83,HS68} \\
$(22,2)$&$100$ &  Higman--Sims graph & yes & \cite{G69,HS68}\\
\hline
\end{tabular}
$PG(2,q)$: projective plane, $GQ(q,q)$: generalized quadrangle, \\
$GH(q,q)$: generalized hexagon, $q$: prime power 
\end{table}

The bounds in Table \ref{tab:1} solve several conjectures of Richey, Shutty, and Stover \cite{RSS}. Richey, Shutty, and Stover prove that $v(3,2)\leq105$, but they note that the largest 3-regular graph with $\lambda_2\leq 2$ they are aware of is the Tutte-Coxeter graph on 30 vertices. They conjectured that $v(3,2)=30$. They show that $v(4,2)\leq77$ and conjecture that the largest 4-regular graph with $\lambda_2\leq2$ is the so-called rolling cube graph on 24 vertices (that is, the bipartite double of the cuboctahedral graph which is the line graph of the $3$-cube). They also conjectured that $v(4,3)=27$ and the largest 4-regular graph with $\lambda_2\leq3$ is the Doyle graph on 27 vertices (see \cite{Doyle,Holt} for a description of this graph). In Table \ref{tab:1} we confirm that $v(3,2)=30$ and the Tutte-Coxeter graph (the incidence graph of $GQ(2,2)$) is, in fact, the unique graph which meets this bound (see \cite[Theorem 7.5.1]{BCN} for uniqueness). However, Table \ref{tab:1} shows that $v(4,2)=35$ (the Odd graph $O_4$) and that $v(4,3)=728$ (the incidence graph of $GH(3,3)$), disproving the latter two conjectures.

Since the order of a graph is an integer, $v(k,\lambda)$ can be bounded above by 
$\lfloor M(k,t,c) \rfloor$. The graphs meeting the bound $M(k,t,c)$ can be maximal under the assumption of a larger second eigenvalue. 
\begin{proposition} \label{prop:large}
Let $\lambda_1$, $\lambda_2$ be the second largest eigenvalues of $T(k,t+1,c_1)$ and $T(k,t,c_2)$, respectively. Suppose there exists a graph which attains the bound $M(k,t,c)$ of Theorem \ref{bound}. Then  
\begin{enumerate}
\item If $c=1$, then $v(k,\lambda_1)=
v(k,\lambda)$ for $c_1>k(k-1)^{t-1}$. 
Moreover if $M(k,t,c)$ is even, and $k$ is odd, then $v(k,\lambda_1)=v(k,\lambda)$ for $c_1>k(k-1)^{t-1}/2$. 
\item If $c>1$,  $v(k,\lambda_2)=
v(k,\lambda)$ for $c_2>c-c^2/(k(k-1)^{t-2}+c)$. 
Moreover if $M(k,t,c)$ is even, and $k$ is odd, then $v(k,\lambda_2)=
v(k,\lambda)$ for $c_2>c-2c^2/(k(k-1)^{t-2}+2c)$. 
\end{enumerate}
\end{proposition}
\begin{proof}
We show only (1) because (2) can be proved similarly.  For $c_1>k(k-1)^{t-1}$, we have
\[
M(k,t,c)=v(k,\lambda) \leq v(k,\lambda_1)\leq \lfloor M(k,t,c_1) \rfloor= M(k,t,c). 
\]
Therefore $v(k,\lambda) = v(k,\lambda_1)$. If $k$ is odd, $v(k,\lambda_1)$ must be even. For $c_1>k(k-1)^{t-1}/2$, we have
\[
M(k,t,c)=v(k,\lambda) \leq v(k,\lambda_1)\leq \lfloor M(k,t,c_1) \rfloor = M(k,t,c)+1. 
\]
Thus if $M(k,t,c)$ is even, then $v(k,\lambda) = v(k,\lambda_1)$. 
\end{proof}
The larger second eigenvalues in Proposition \ref{prop:large} are calculated in Table \ref{tab:2}. 
The graphs in Table \ref{tab:2} meet $v(k,\lambda)$ for any $\lambda_2\leq \lambda<\lambda'$, where $\lambda'$ is the largest zero of $f(x)$ in the table. 
\begin{table}[b!]
\caption{Graphs meeting $v(k,\lambda)$ for $\lambda_2 \leq \lambda<\lambda'$}\label{tab:2}
\begin{tabular}{|c|c|c|c|c|} 
\hline
 Graph &$t$&$c$&$f(x)$&$\lambda'$ \\
\hline 
  $K_{k+1}$ ($k$: even)&2 &1 &$x^2- \left(k-k^2\right)x+k^2-2 k$ &\\ 
  $K_{k+1}$  ($k$: odd)&2 &1 &$ 2 x^2-\left(k-k^2\right)x+k^2-3 k$ &\\ 
  $K_{k,k}$ ($k$: even)  & 3& $k$& $x^2-(1-k) x-1$&\\
  $K_{k,k}$ ($k$: odd)  & 3& $k$& $(k+1) x^2+\left(k^2-k\right) x-2 k$&\\
 $PG(2,q)$ ($q+1$: even) & $4$&$q+1$ &$\left(q^2+1\right) x^3+\left(q^3+q^2\right) x^2$&\\
& & & $+\left(-q^3-2 q-1\right) x-q^4-q^3$ & \\
 $PG(2,q)$ ($q+1$: odd) & $4$&$q+1$ & $+\left(q^2+2\right) x^3+\left(q^3+q^2\right) x^2$&\\
& & &$+\left(-q^3-4 q-2\right) x-q^4-q^3$ & \\
 $GQ(q,q)$ ($q+1$: even) &5 & $q+1$&$\left(-q^2+q-1\right) x^4-q^3 x^3$ &\\
& & &$+\left(2 q^3-2 q^2+2 q+1\right) x^2$ & \\ 
& & &$+2 q^4 x-q$ & \\
 $GQ(q,q)$ ($q+1$: odd) &5 & $q+1$&$\left(-q^3-2\right) x^4+\left(-q^4-q^3\right) x^3$ &\\
& & &$+\left(2 q^4+6 q+2\right) x^2+\left(2 q^5+2 q^4\right) x$ & \\
& & &$-2 q^2-2 q$ & \\
 
$GH(q,q)$ ($q+1$: even)& 7&$q+1$ &$\left(-q^4+q^3-q^2+q-1\right) x^6$ & \\
& & &$+\left(4 q^5-4 q^4+4 q^3-4 q^2+4 q+1\right) x^4$ & \\
& & &$+\left(-3 q^6+3 q^5-3 q^4+3 q^3-3 q^2-3 q\right) x^2$ & \\
& & &$-q^5 x^5+4 q^6 x^3-3 q^7 x+q^2$ & \\
 $GH(q,q)$ ($q+1$: odd)& 7&$q+1$ &$\left(-q^5-2\right) x^6+\left(-q^6-q^5\right) x^5$ & \\
& & &$+\left(4 q^6+10 q+2\right) x^4+\left(4 q^7+4 q^6\right) x^3$ & \\
& & & $+\left(-3 q^7-12 q^2-6 q\right) x^2$& \\
& & &$+\left(-3 q^8-3 q^7\right) x+2 q^3+2 q^2$ & \\
 Petersen &3& 1 &$x^3+ 12 x^2+ 7 x-24 $&1.11207\\
Odd graph $O_4$&4&2&$19 x^3+36 x^2-97 x-108$ & 2.02156\\
 Hoffman--Singleton & 3&1&$ x^3+ 126 x^2+ 113 x-756$&2.02845\\
 Clebsch   &3&2& $3 x^2+5 x-10$&$
1.1736 $\\
 Gewirtz  &3&2&$23 x^2+45 x-185$&$
2.02182$\\
 $M_{22}$  &3 &4&$61 x^2+240 x-736$&$
 2.02472$\\
  Higman--Sims  &3&6&$13x^2+77x-209$ &
$
2.0232$\\
\hline
\end{tabular}\\
$\lambda'$ is the largest zero of $f(x)$ 
\end{table}

By Theorem~\ref{bound}, we can obtain an 
alternative proof of the theorem due to 
Alon and Boppana, and Serre (see \cite{A,C,F,HLW,KS,LPS,Mo,Ni2,Ni1,Serre1} for more details).
\begin{corollary}[Alon--Boppana, Serre]
For given $k$, $\lambda<2 \sqrt{k-1}$, there exist finitely many $k$-regular graphs whose 
second largest eigenvalue is at most $\lambda$.  
\end{corollary}
\begin{proof}
The second largest eigenvalue $\lambda_2(t)$ of $T(k,t,1)$ is equal to the largest zero of 
$G_{t-1}$.  The zero is expressed by
$\lambda_2(t)=2\sqrt{k-1} \cos \theta$, where $\theta$ is less than $\pi/(t-1)$ \cite[Section III.3]{BIb}. 
This implies that there exists a sufficiently large $t'$ such that $\lambda_2(t')>\lambda$.  
Therefore we have
\[
v(k, \lambda)\leq v(k, \lambda_2(t')) \leq 1+\sum_{i=0}^{t'-2}k(k-1)^i. \qquad \qedhere 
\]
\end{proof}


\section{Second largest eigenvalue $1$}\label{vk1}

In this section, we classify the graphs meeting $v(k,1)$. 
The complement of a regular graph with second eigenvalue at most $1$ has smallest eigenvalue at least $-2$. The structure of such graph is obtained from a subset of a root system, and it is characterized as a line graph except for sporadic examples \cite[Theorem 3.12.2]{BCN}. The following theorem is immediate by \cite[Theorem 3.12.2]{BCN}. 
\begin{theorem} \label{thm:BCN}
Let $G$ be a connected regular graph with $v$ vertices, valency $k$, and second largest eigenvalue 
 at most $1$. Then one of the following holds: 
\begin{enumerate}
\item $G$ is the complement of the line graph of a regular  
or a bipartite semiregular connected graph. \label{bcn1}
\item $v=2(k-1) \leq 28$, and $G$ is a subgraph of the complement of $E_7(1)$, switching-equivalent
to the line graph of a graph $\Delta$ on eight vertices, where all valencies of $\Delta$
have the same parity $($graphs nos.\ $1$--$163$ in Table~$9.1$ in {\rm \cite{BCS}}$)$. \label{bcn2}
\item $v=3(k-1) \leq 27$, and $G$ is a subgraph of the complement  of the Schl$\ddot{a}$fli graph $($graphs nos.\ $164$--$184$ in Table~$9.1$ in {\rm \cite{BCS}}$)$.  \label{bcn3}
\item $v=4(k-1) \leq 16$, and $G$ is a subgraph of the complement 
of the Clebsch graph  $($graphs nos.\ $185$--$187$ in Table~$9.1$ in {\rm \cite{BCS}}$)$. \label{bcn4}
\end{enumerate} 
\end{theorem}
The following theorem shows the classification of graphs meeting $v(k,1)$. Note that this result will show that $v(k,1)=2k+2$ for $k$ large whereas Theorem \ref{bound} would give a larger upper bound for $v(k,1)$.
\begin{theorem}\label{main5}
Let $G$ be a connected $k$-regular graph with second largest eigenvalue at most $1$, with 
$v(k,1)$ vertices. 
Then the following hold:   
\begin{enumerate}
\item  $v(2,1) = 6$, and $G$ is the hexagon. \label{1}
\item  $v(3,1) = 10$,  and $G$ is the Petersen graph. \label{2}
\item  $v(4,1) = 12$, and $G$ is the complement of the graph no.\ $186$ in Table~$9.1$ in {\rm \cite{BCS}}. \label{3} 
\item  $v(5,1) = 16$, and $G$ is the Clebsch graph. \label{4}
\item $v(6,1)=15$, and $G$ is the complement of the line graph of the complete graph with $6$ vertices, or the complement of one of the graphs 
nos.\ $171$--$176$ in Table~$9.1$ in {\rm \cite{BCS}}.  \label{5}
\item $v(7,1)=18$,  and $G$ is the complement of one of the graphs nos.\ $177$--$180$ in Table~$9.1$ in {\rm \cite{BCS}}.  \label{6}
\item $v(8,1)=21$,  and $G$ is the complement of one of the graphs nos.\ $181$, $182$ in Table~$9.1$ in {\rm \cite{BCS}}.  \label{7}  
\item $v(9,1)=24$,  and $G$ is the complement of the graph no.\ $183$ in Table~$9.1$ in {\rm \cite{BCS}}. 
 \label{8} 
\item $v(10,1)=27$, and $G$ is the complement of the Schl$\ddot{a}$fli graph. \label{9}
\item $v(k, 1) = 2k +2$ for $k \geq 11$, and $G$ is the complement of the line graph of $K_{2,k+1}$.  \label{10}  
\end{enumerate}
\end{theorem}
\begin{proof}
(\ref{1}): A connected $2$-regular graph is an $n$-cycle, whose eigenvalues are $2\cos (2 \pi j /n)$ ($j=0,1,\ldots, n-1$). This implies (\ref{1}).

(\ref{2}), (\ref{4}):  By Theorem \ref{bound} for $T(k,3,(k-1)/2)$, we have $v(k,1)\leq 3k+1$. 
The two graphs are unique graphs attaining this bound (see \cite[Theorem~10.6.4]{GRb} and \cite{HS,N}).

(\ref{10}): The complement of the line graph of $K_{2,k+1}$ is of degree $k$ and has $2k+2$ vertices for
any $k$. We will prove that there exists no graph with at least $2k+2$ vertices except for these graphs for $k\geq 11$. 
In the case of Theorem \ref{thm:BCN} (\ref{bcn3}) (\ref{bcn4}),  we have no graph for $k\geq 11$. 
In the case of Theorem \ref{thm:BCN} (\ref{bcn2}), trivially $v=2(k-1)<2k+2$. We consider the case of 
Theorem \ref{thm:BCN} (\ref{bcn1}).  
Let $G$ be the complement of the line graph of a $t$-regular graph with $u$ vertices. Then $G$
is of degree $k=(u/2-2)t+1$, and has $v=ut/2$ vertices. Therefore $v=ut/2=u(k-1)/(u-4)\leq
2(k-1)< 2k+2$ because $u\geq 8$ for $k\geq 11$. 
Let $G$ be the complement of the line graph 
of a bipartite semiregular connected graph $(V_1,V_2,E)$. 
Let $|V_i|=u_i$
and  the degree of $x \in V_i$ be $t_i$, where we suppose $t_1\geq t_2$.  Then
$G$ is of degree $k=(u_1-1)t_1-t_2+1\geq (u_1-2)t_1+1$, and has $v=u_1t_1$ vertices. 
If $u_1=1$ holds, then $G$ has no edge.  
For $u_1>3$, it is satisfied that 
\begin{equation}\label{eq:v1}
v\leq\left(1+\frac{2}{u_1-2}\right)(k-1)\leq 2(k-1)<2k+2
\end{equation} 
for any $k$. 
For $u_1=3$, we have 
$t_2 \leq u_1 = 3$ and  
\begin{equation} \label{eq:v2}
v=3t_1=\frac{3}{2}(k+t_2-1) \leq \frac{3}{2}(k+2)<2k+2
\end{equation}
 for $k>2$. 
 For $u_1=2$, similarly $t_2 \leq u_1=2$ and
\begin{equation}\label{eq:v3}
 v=2 t_1=2(k+t_2-1) \leq 2k+2
\end{equation}
for any $k$, with equality only if  $t_1=k+1$, $t_2=2$, $u_1=2$ and $u_2=k+1$. 
Thus  (\ref{10}) holds. 

(3), (\ref{5})--(\ref{9}): 
Every candidate of maximal graphs comes from Theorem \ref{thm:BCN} (\ref{bcn3}) or (\ref{bcn4}) except for the case of the complete graph in (\ref{5}). We prove that there does not exist a
larger graph which comes from Theorem \ref{thm:BCN} (\ref{bcn1}). By inequalities \eqref{eq:v1}--\eqref{eq:v3}, the complement of 
the line graph of a bipartite semiregular graph is not maximal for $k>2$. 
 We consider the case of the complements of the line graphs of $t$-regular graphs with $u$ vertices. 
Since  
$v=k-1+2t$ is at least $12,15,18,21,24,27$, we have $u-1 \geq t\geq 5,5,6,7,8,9$ for $k=4,6,7,8,9,10$, respectively. 
Therefore $k=(u/2-2)t+1 \geq (t-2)(t-1)/2 \geq 6,6,10,15,21,28$ for $k=4,6,7,8,9,10$, 
respectively. 
The only parameter $(v,k,u,t)=(15,6,6,5)$ satisfies the conditions and it corresponds to the case of the complete graph in (\ref{5}). 
\end{proof}

\section{Other Values of $v(k,\lambda)$}

When no graph meets the bound given by Theorem \ref{bound}, other techniques may be necessary to find $v(k,\lambda)$. However, the bound is still useful in reducing the size of graphs which must be checked. In this section we describe several tools which we will use (Lemma \ref{subh1} and Lemma \ref{subh}), and then find $v(k,\lambda)$ in a few more cases (Proposition \ref{cubicsqrt2}, Proposition \ref{quarticsqrt5}, Proposition \ref{oneptnine}). 

Let $n(k,g)$ denote the minimum possible number of vertices of a $k$-regular graph with girth $g$. A $(k,g)$-cage is a graph which attains this minimum. The following lower bound on $n(k,g)$ due to Tutte \cite{WT} will be useful.
\begin{lemma}\label{gbound}
Define $n_l(k,g)$ by
\[
n_l(k,g)=\begin{cases}
\frac{k(k-1)^{(g-1)/2}-2}{k-2}&\mbox{if }g\mbox{ is odd,}\\
\frac{2(k-1)^{g/2}-2}{k-2}&\mbox{if }g\mbox{ is even.}\\
\end{cases}
\]
Then $n(k,g)\geq n_l(k,g)$.
\end{lemma}

The following lemma is easily verified.
\begin{lemma}\label{specradlem}'
Each of the graphs in Figure \ref{specrad} has spectral radius greater than 2.
\end{lemma}

\setlength{\unitlength}{1.5pt}
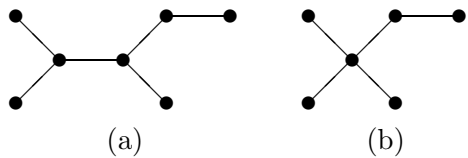
\begin{figure}[h!]
\centering
\begin{subfigure}[b]{.2\textwidth}
\centering
\begin{picture}(55,23)(0,0)
\put(0,0){\circle*{3}}
\put(11,11){\circle*{3}}
\put(0,22){\circle*{3}}
\put(27,11){\circle*{3}}
\put(38,0){\circle*{3}}
\put(38,22){\circle*{3}}
\put(54,22){\circle*{3}}
\put(0,0){\line(1,1){11}}
\put(11,11){\line(-1,1){11}}
\put(11,11){\line(1,0){16}}
\put(27,11){\line(1,1){11}}
\put(27,11){\line(1,-1){11}}
\put(38,22){\line(1,0){16}}
\end{picture}
\caption{\label{specrad6}}
\end{subfigure}
~
\begin{subfigure}[b]{.17\textwidth}
\centering
\begin{picture}(39,23)(0,0)
\put(0,0){\circle*{3}}
\put(11,11){\circle*{3}}
\put(0,22){\circle*{3}}
\put(22,22){\circle*{3}}
\put(22,0){\circle*{3}}
\put(38,22){\circle*{3}}
\put(0,0){\line(1,1){11}}
\put(11,11){\line(-1,1){11}}
\put(22,22){\line(1,0){16}}
\put(11,11){\line(1,1){11}}
\put(11,11){\line(1,-1){11}}
\end{picture}
\caption{\label{specrad7}}
\end{subfigure}
\caption{Graphs with spectral radius greater than 2.}\label{specrad}
\end{figure}

For a graph $G$, a vertex $v\in V(G)$, and a subset $U\subset V(G)$, define the distance $\dist(v,U)=\min_{u\in U}\dist(u,v)$. For an induced subgraph $H$ of $G$, let $\gamh i$ and $\gamhg i$ be the sets of vertices in $G$ at distance exactly $i$ and at least $i$ from $V(H)$ in $G$, respectively. Let $\rad G$ and $\deg G$ denote the spectral radius and average degree of $G$, respectively. Note that $\deg G\leq \rad G$.
\begin{lemma}\label{subh1}
Suppose $G$ is a connected, $k$-regular graph with second largest eigenvalue $\lambda_2(G)\leq\lambda<k$, and $H$ is an induced subgraph of $G$ with $\deg H\geq\lambda$. Then for the subgraph $K$ induced by $\gamhg2$ we have $\deg K\leq\lambda$, with equality only if $\deg H=\lambda_2(G)=\lambda$.
\end{lemma}
\begin{proof}
Consider the quotient matrix $Q$ of the partition $\{V(H),\gamh1,\gamhg2\}$ of $V(G)$. We have
\[
Q=\begin{pmatrix}
\alpha & k-\alpha & 0\\
\gamma & k-(\gamma+\epsilon) & \epsilon\\
0 & k-\beta & \beta
\end{pmatrix},
\]
where $\alpha=\deg H$, $\beta=\deg K$, and $\gamma$ and $\epsilon$ are the average numbers of neighbors in $H$ and $K$, respectively, of the vertices in $\gamh1$. The eigenvalues of $Q$ interlace those of $G$ (see \cite[Corollary~2.5.4]{BH}), so we must have $\lambda_2(Q)\leq\lambda_2(G)\leq\lambda$. It is straightforward to verify that $\lambda_1(Q)=k$ and
\be\label{l2}
\lambda_2(Q)=\frac12\left(\alpha+\beta-(\gamma+\epsilon)+\sqrt{\Delta}\right),
\ee
where $\Delta=(\alpha+\beta-(\gamma+\epsilon))^2-4(\alpha\beta-\beta\gamma-\alpha\epsilon)$. By hypothesis we have $\alpha\geq\lambda$. If also $\beta\geq\lambda$, then we find that $\alpha=\beta=\lambda_2(Q)=\lambda$, as we will prove below.

Indeed, if both $\alpha>\lambda$ and $\beta>\lambda$, then by Cauchy interlacing \cite[Proposition~3.2.1]{BH} $\lambda_2(G)\geq\lambda_2(H+K)>\lambda$, where $H+K$ is the disjoint union of $H$ and $K$, a contradiction. Suppose $\alpha\geq\lambda$ and $\beta\geq\lambda$. If $\alpha=\beta=\lambda$, then \eqref{l2} becomes $\lambda_2(Q)=\lambda$. Otherwise we must have $\alpha>\beta=\lambda$ or $\beta>\alpha=\lambda$. If $\sqrt{\Delta}\geq \gamma+\epsilon$, then clearly $\lambda_2(Q)>\lambda$, a contradiction. If $\sqrt{\Delta}< \gamma+\epsilon$, then $\Delta< (\gamma+\epsilon)^2$, which implies $(\alpha-\beta)(\alpha-\beta+2(\epsilon-\gamma))<0$. Thus we have either $\alpha>\beta$ and $\epsilon<\gamma-\frac12(\alpha-\beta)$, or $\beta>\alpha$ and $\gamma<\epsilon-\frac12(\beta-\alpha)$. Suppose the former is true. Then $\beta=\lambda$ and we can write $\alpha=\beta+s=\lambda+s$ and $\epsilon=\gamma-\frac{s}{2}-t$ for some $s,t>0$. Then \eqref{l2} becomes
\[
\lambda_2(Q)=\frac14\left(4\lambda-4\gamma+3s+2t+\sqrt{\Delta'}\right),
\]
where $\Delta'=16\gamma^2+(s-2t)^2-8\gamma(s+2t)$. If $\sqrt{\Delta'}>4\gamma-3s-2t$, then clearly $\lambda_2(Q)>\lambda$, a contradiction. If $\sqrt{\Delta'}\leq4\gamma-3s-2t$, then $\Delta'\leq(4\gamma-3s-2t)^2$, which implies $\gamma\leq\frac{s}{2}+t$. However, this implies $\epsilon=\gamma-\frac{s}{2}-t\leq0$, a contradiction. If $\beta>\alpha$ and $\gamma<\epsilon-\frac12(\beta-\alpha)$, the same argument holds (simply swap the roles of $\alpha$ and $\beta$ and of $\gamma$ and $\epsilon$ in the above argument). Thus we cannot have $\alpha\geq\lambda$ and $\beta\geq\lambda$ unless $\alpha=\beta=\lambda$, so we must have $\beta<\lambda$ or $\alpha=\beta=\lambda_2(Q)=\lambda$.
\end{proof}
\begin{lemma}\label{subh}
Suppose $G$ is a connected, $k$-regular graph with second largest eigenvalue $\lambda_2(G)\leq\lambda<k$. If $G$ contains an induced subgraph $H$ on $s$ vertices with $t$ edges and either $\deg H\geq\lambda$ or $\rad H>\lambda$, then
\be\label{vertbound}
\abs{V(G)}\leq s+\frac{2k-\lambda-1}{k-\lambda}(ks-2t).
\ee
\end{lemma}
\begin{proof}
Since $G$ is $k$-regular, there are $ks-2t$ edges from $H$ to $\gamh1$, which implies $\abs{\gamh1}\leq ks-2t$. We will show that $\abs{\gamhg2}\leq \frac{k-1}{k-\lambda}\abs{\gamh1}$, which completes the proof that \eqref{vertbound} holds.

First, note that each vertex in $\gamh1$ has a neighbor in $H$, so each such vertex has at most $k-1$ neighbors in $\gamhg2$. Then there are at most $(k-1)\abs{\gamh1}$ edges from $\gamh1$ to $\gamhg2$. If $\deg H\geq\lambda$ then by Lemma \ref{subh1} we have $\deg K\leq\lambda$, where $K$ is the subgraph induced by $\gamhg2$. If not, then $\rad H>\lambda$, so $\rad K\leq\lambda$ (and so also $\deg K\leq\lambda$) by eigenvalue interlacing. Since $G$ is $k$-regular, this implies that the average number of neighbors in $\gamh1$ of the vertices in $\gamhg2$ is at least $k-\lambda$, so there are at least $(k-\lambda)\abs{\gamhg2}$ edges from $\gamhg2$ to $\gamh1$. This completes the proof.
\end{proof}

\begin{proposition}\label{cubicsqrt2}
If $G$ is a connected, 3-regular graph with $\lambda_2(G)>1$, then $\lambda_2(G)\geq\sqrt{2}$, with equality if and only if $G$ is the Heawood graph.
\end{proposition}
\begin{proof}
We have already seen in Table \ref{tab:1} that $v(3,\sqrt{2})=14$ and the Heawood graph (the incidence graph of $PG(2,2)$) is the unique graph meeting this bound. Thus we only need to show that no 3-regular graph has second eigenvalue between 1 and $\sqrt{2}$. Suppose $G$ is a 3-regular graph with $1<\lambda_2(G)<\sqrt{2}$. We will show that this yields a contradiction. We have immediately that $\abs{V(G)}<14$. Since $G$ is 3-regular, this implies $\abs{V(G)}\leq12$.

We note that the average degree of any cycle is $2>\sqrt{2}>\lambda_2(G)$. If $G$ has girth 3, then Lemma \ref{subh} implies $\abs{V(G)}\leq\frac{6}{7}(\sqrt{2}+10)\approx9.78$. Since $G$ is 3-regular, this implies $\abs{V(G)}\leq8$. Lemma \ref{gbound} implies that a graph with girth more than 5 has at least 14 vertices, so $G$ has girth at most 5.

We partition the vertices of $G$ by $P_1=\{V(H),\gamh1,\gamhg2\}$, where $H$ is a subgraph of $G$ isomorphic to $C_m$, where $m\in \{3,4,5\}$ is the girth of $G$. This partition has quotient matrix $Q$ given by
\[
Q=\begin{pmatrix}
2 & 1 & 0\\
\gamma & 3-(\alpha+\gamma) & \alpha\\
0 & \beta & 3-\beta
\end{pmatrix},
\]
where $\gamma\abs{\gamh1}=m$ (by counting edges from $H$ to $\gamh1$) and $\alpha\abs{\gamh1}=\beta\abs{\gamhg2}$ (by counting edges from $\gamh1$ to $\gamhg2$).

We first suppose $G$ has girth 3. Then $4\leq\abs{V(G)}\leq8$. If $\abs{V(G)}=4$, then $G\cong K_4$, and we have $\lambda_2(G)=-1$. If $\abs{V(G)}=6$, it is straightforward to show that $G\cong C_3\square K_2$, where $\square$ denotes the graph Cartesian product, and we have $\lambda_2(G)=1$. Either case is a contradiction.

If $\abs{V(G)}=8$ then $\gamh1$ has 2 or 3 vertices. If $\abs{\gamh1}=2$, then we have $\abs{\gamhg2}=3$, $\gamma=3/2$, and depending on whether there is an edge in $\gamh1$ or not we have
$\alpha=1/2$ or $3/2$, $\beta=1/3$ or 1, and $\lambda_2(Q)=\frac{1}{3}(\sqrt{13}+4)\approx2.54$ or 2, respectively. Either case is a contradiction.
If $\abs{\gamh1}=3$, then $\abs{\gamhg2}=2$, $\gamma=1$, and  
depending on whether there is an edge in $\gamhg2$ or not we have $\beta=2$ or $3$, $\alpha=4/3$ or 2, and $\lambda_2(Q)=5/3$ or $\frac{1}{2}(\sqrt{17}-1)\approx1.56$, respectively. Either case is a contradiction. Thus $G$ cannot have girth 3.

Suppose $G$ has girth 4. Then we have $6\leq\abs{V(G)}\leq12$. If $\abs{V(G)}=6$, then $G\cong K_{3,3}$ and we have $\lambda_2(G)=0$. If $\abs{V(G)}=8$, then it is straightforward to verify that $G$ must either be the 3-cube $Q_3$ or the graph in Figure \ref{almostcube}.
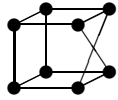
\begin{figure}[h!]
\centering
\begin{picture}(25,22)(0,0)
\put(0,0){\circle*{3}}
\put(16,0){\circle*{3}}
\put(0,16){\circle*{3}}
\put(16,16){\circle*{3}}
\put(8,4){\circle*{3}}
\put(24,4){\circle*{3}}
\put(8,20){\circle*{3}}
\put(24,20){\circle*{3}}
\put(0,0){\line(1,0){16}}
\put(0,0){\line(0,1){16}}
\put(16,16){\line(-1,0){16}}
\put(16,16){\line(2,-3){8}}
\put(0,0){\line(2,1){8}}
\put(16,0){\line(2,1){8}}
\put(0,16){\line(2,1){8}}
\put(16,16){\line(2,1){8}}
\put(8,4){\line(1,0){16}}
\put(8,4){\line(0,1){16}}
\put(24,20){\line(-1,0){16}}
\put(16,0){\line(2,5){8}}
\end{picture}
\caption{A 3-regular graph on 8 vertices with girth 4.}
\label{almostcube}
\end{figure}
In either case we have $\lambda_2(G)=1$, a contradiction. If $\abs{V(G)}=10$, then $\gamh1$ has 2, 3, or 4 vertices. If $\abs{\gamh1}=2$, then $\abs{\gamhg2}=4$, $\gamma=2$, $\alpha=1$, $\beta=1/2$, and
$\lambda_2(Q)=\frac14(\sqrt{41}+3)\approx2.35$, a contradiction. If $\abs{\gamh1}=3$, then $\abs{\gamhg2}=3$, $\gamma=4/3$, and 
$\alpha=\beta$. Then $\alpha\leq5/3$ (since $3-(\alpha+\gamma)\geq0$) implies $\beta\leq5/3$, which implies $\gamhg2$ has at least 2 edges. Since $G$ has girth 4, $\gamhg2$ cannot have 3 edges, so $\gamhg2$ has exactly 2 edges, $\alpha=\beta=5/3$, and $\lambda_2(Q)=\frac12(\sqrt{241}+7)\approx1.88$, a contradiction. If $\abs{\gamh1}=4$, then $\abs{\gamhg2}=2$, $\gamma=2$, and 
depending on whether there is an edge in $\gamhg2$ or not we have $\beta=2$ or 3, $\alpha=1$ or $3/2$, and $\lambda_2(Q)=\frac12(\sqrt{5}+1)\approx1.62$ or 3/2, respectively. Either case is a contradiction. If $\abs{V(G)}=12$, then $\gamh1$ must be a coclique on 4 vertices (otherwise there are at most 6 edges from $\gamh1$ to $\gamhg2$, so Lemma \ref{subh1} implies $\abs{\gamhg2}<6/(3-\sqrt{2})\approx3.78$, which implies $\abs{V(G)}<11.78$, a contradiction). Then we have $\abs{\gamh1}=\abs{\gamhg2}=4$, $\gamma=1$, 
 $\alpha=\beta=2$, and
 $\lambda_2(Q)=\sqrt{3}$. This is a contradiction, so $G$ cannot have girth 4.

Suppose $G$ has girth 5. Then $10\leq\abs{V(G)}\leq12$. The Petersen graph with 10 vertices and $\lambda_2=1$ is the unique $(3,5)$-cage (see \cite{HS}), so $G$ must have 12 vertices. Note we must have $\abs{\gamh1}=5$ and $\gamma=1$, since vertices in $H$ cannot have common neighbors outside of $H$. Since $\abs{V(G)}=12$, we have $\abs{\gamhg2}=2$, and 
depending on whether there is an edge in $\gamhg2$ or not we have $\beta=2$ or 3, $\alpha=4/5$ or $6/5$, and $\lambda_2(Q)=\frac15(2\sqrt{6}+3)\approx1.58$ or $\frac{1}{10}(\sqrt{241}-1)\approx1.45$, respectively. Either case is a contradiction.

Thus $G$ cannot exist as described, which completes the proof.
\end{proof}

\begin{proposition}\label{quarticsqrt5}
If $G$ is a connected, 4-regular graph with $\lambda_2(G)>1$, then $\lambda_2(G)\geq\sqrt{5}-1$, with equality if and only if $G$ is either the graph in Figure \ref{sqrt5} or the circulant graph $\mbox{Ci}_{10}(1,4)$ (the Cayley graph of $(\zee_{10},+)$ with generating set $\{\pm 1,\pm 4\}$).
\end{proposition}
\begin{figure}[h]
\centering
\includegraphics[width=.15\linewidth]{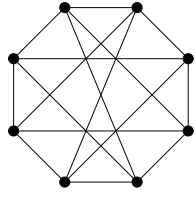}
\caption{The 4-regular graph $G$ on 8 vertices with $\lambda_2(G)=\sqrt{5}-1$.}\label{sqrt5}
\end{figure}
\begin{proof}
It is straightforward to verify that the second eigenvalue of $T(4,3,(4-(\sqrt{5}-1)^2)/\sqrt{5})=\sqrt{5}-1$ and $M(4,3,(4-(\sqrt{5}-1)^2)/\sqrt{5})=5+12\sqrt{5}/(4-(\sqrt{5}-1)^2)\approx15.85$, so by Theorem \ref{bound} we have $v(4,\sqrt{5}-1)\leq15$.
%
We checked by computer all 4-regular graphs on at most 15 vertices and found that, in each case where $\lambda_2(G)>1$, we have $\lambda_2(G)\geq\sqrt{5}-1$, with equality if and only if $G$ is either the graph in Figure \ref{sqrt5} or the circulant graph $\mbox{Ci}_{10}(1,4)$.
%
%
\end{proof}
The previous result and Theorem \ref{main5} part (iii) imply that $v(4,\sqrt{5}-1)=12$. It would be interesting to find a proof of Proposition \ref{quarticsqrt5} which does not require a computer search. For the proof above the computer must check 906,331 graphs.

Richey, Shutty, and Stover \cite{RSS} conjectured that $v(3,1.9)=18$. We confirm this conjecture, and show that there are exactly two graphs meeting this bound.
\begin{proposition}\label{oneptnine}
If $G$ is a connected, 3-regular graph with second largest eigenvalue $\lambda_2(G)\leq1.9$, then $\abs{V(G)}\leq18$, with equality if and only if $G$ is the Pappus graph (see Figure \ref{pappus}) or the graph in Figure \ref{onept9}.
\end{proposition}
\begin{figure}[h]
\centering
\begin{subfigure}{.3\linewidth}
\centering
\includegraphics[width=.66\textwidth]{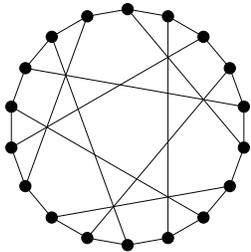}
\caption{The Pappus graph with second eigenvalue $\sqrt3$.\\ \phantom{words}}\label{pappus}
\end{subfigure}
~
\begin{subfigure}{.5\linewidth}
\centering
\includegraphics[width=.4\textwidth]{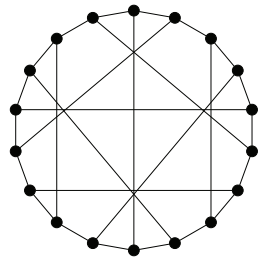}
\caption{A graph with $\lambda_2=\gamma\approx1.8662$, the largest root of $f(x)=x^3+2x^2-4x-6$.}\label{onept9}
\end{subfigure}
\caption{The 3-regular graphs on 18 vertices with $\lambda_2<1.9$.}
\end{figure}
\begin{proof}
It is straightforward to verify that the second eigenvalue of $T(3,4,2641/3510)=19/10=1.9$ and $M(3,4,2641/3510)=68530/2641\approx25.95$, so by Theorem \ref{bound} we have $v(3,1.9)\leq25$. Since $G$ is 3-regular, this implies $v(3,1.9)\leq24$.
We note again that any cycle has spectral radius 2. 
Then, by Lemma \ref{subh}, if $G$ has girth 3, 4, 5, or 6, then $G$ has at most 11.45, 15.27, 19.09, or 22.91 vertices, respectively. Since $G$ is 3-regular, this implies $G$ has at most 10, 14, 18, or 22 vertices, respectively. A 3-regular graph of girth 8 has at least 30 vertices by Lemma \ref{gbound} (or note that the Tutte-Coxeter graph is the unique (3,8)-cage, see \cite{WT2,WT}). Thus, we have shown that a 3-regular graph $G$ with $\lambda_2(G)\leq1.9$ and more than 18 vertices must have girth 6 or 7.

If $G$ has girth 7, we note that the McGee graph on 24 vertices is the unique (3,7)-cage (see \cite[p.209]{BCN} or \cite{WM,WT}), so $G$ must be the McGee graph. Since the McGee graph has second eigenvalue 2, we have proved that $G$ does not have girth 7.

Now, if $G$ has more than 18 vertices then $G$ must have girth $6$ and at most 22 vertices. Among 3-regular graphs, we checked by computer the 32 graphs with girth 6 on 20 vertices and the 385 graphs with girth 6 on 22 vertices and found that each has second eigenvalue more than 1.9. Thus $G$ has at most 18 vertices. If $G$ has 18 vertices, then $G$ must have girth 5 or 6. Among 3-regular graphs, we checked by computer the 450 graphs with girth 5 on 18 vertices and found that each has second eigenvalue more than 1.9. We checked the 5 graphs with girth 6 on 18 vertices and found that all but two of them have second eigenvalue more than 1.9. The exceptions were the Pappus graph with second eigenvalue $\sqrt{3}$ and the graph in Figure \ref{onept9} with second eigenvalue $\gamma$, where $\gamma\approx1.8662$ is the largest root of $f(x)=x^3+2x^2-4x-6$.
\end{proof}
Note that this implies $v(3,\sqrt{3})=18$ and $v(3,\gamma\approx1.8662)=18$ (and, of course, $v(3,1.9)=18$). It would be nice to find a proof of Proposition \ref{oneptnine} that does not require a computer search.

\section{Final Remarks}

We conclude the paper with some questions and problems for future research.
\begin{problem}
Determine $v(k,\sqrt{k})$ for $k\geq 3$.
\end{problem}
We have $\lambda_2(T(k,4,k-\sqrt{k}))=\sqrt{k}$ and $M(k,4,k-\sqrt{k})=2 k^2+k^{3/2}-k-\sqrt{k}+1$, which yields 
\[
v(k,\sqrt{k}) \leq 2 k^2+k^{3/2}-k-\sqrt{k}+1.
\]
The Odd graph $O_4$ meets this bound (see Table \ref{tab:1}). We do not know what other graphs, if any, meet this bound. Odd graphs, in general, do not have $T(k,t,c)$ as a quotient matrix.



\begin{problem}\label{psqrt2}
Determine $v(k, \sqrt{2})$ for $k\geq 3$.
\end{problem}
Recall that for $k=3$ we have $v(3,\sqrt2)=14$ and the Heawood is the unique graph meeting this bound. 
For $k>3$ we note that Lemma \ref{subh} with $H=K_3$ implies that a graph $G$ with $\lambda_2(G)\leq\sqrt{2}$ and girth 3 satisfies $\abs{V(G)}\leq3(k-1)\left(1+\frac{k-2}{k-\sqrt{2}}\right)$, and Lemma \ref{subh} with $H=K_{1,3}$ implies that such a graph with girth more than 3 satisfies $\abs{V(G)}\leq4+2(2k-3)\left(1+\frac{k-1}{k-\sqrt{2}}\right)$ (note that in both cases we have $\rad{H}>\lambda_2(G)$). Combining this with Lemma \ref{gbound} allows one to restrict the search to graphs with certain girth. For $k\geq7$, $n_l(k,g)$ is larger than these bounds unless the girth is at most 4, and for $k=4$, 5, or 6 $n_l(k,g)$ is larger than these bounds unless the girth is at most 5. Thus the graphs sought in Problem \ref{psqrt2} must have girth at most 5 for $k=4,5,6$ and girth at most 4 for $k\geq7$.

\begin{problem}
Among regular graphs, what is the smallest second eigenvalue larger than 1?
\end{problem}
Yu \cite{Y} found a 3-regular graph $G$ on 16 vertices (see Figure \ref{largestleast}) 
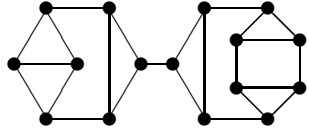
\begin{figure}[h]
\centering
\begin{picture}(74,30)(0,0)
\put(0,14){\circle*{3}}
\put(8,0){\circle*{3}}
\put(8,28){\circle*{3}}
\put(16,14){\circle*{3}}
\put(24,0){\circle*{3}}
\put(24,28){\circle*{3}}
\put(32,14){\circle*{3}}
\put(40,14){\circle*{3}}
\put(48,0){\circle*{3}}
\put(48,28){\circle*{3}}
\put(56,8){\circle*{3}}
\put(56,20){\circle*{3}}
\put(64,0){\circle*{3}}
\put(64,28){\circle*{3}}
\put(72,8){\circle*{3}}
\put(72,20){\circle*{3}}
\put(0,14){\line(3,5){8}}
\put(0,14){\line(3,-5){8}}
\put(0,14){\line(1,0){16}}
\put(16,14){\line(-3,-5){8}}
\put(16,14){\line(-3,5){8}}
\put(8,0){\line(1,0){16}}
\put(8,28){\line(1,0){16}}
\put(24,28){\line(0,-1){28}}
\put(32,14){\line(-3,-5){8}}
\put(32,14){\line(-3,5){8}}
\put(32,14){\line(1,0){8}}
\put(40,14){\line(3,5){8}}
\put(40,14){\line(3,-5){8}}
\put(48,28){\line(0,-1){28}}
\put(48,0){\line(1,0){16}}
\put(48,28){\line(1,0){16}}
\put(56,20){\line(0,-1){12}}
\put(56,20){\line(1,1){8}}
\put(56,8){\line(1,-1){8}}
\put(56,20){\line(1,0){16}}
\put(56,8){\line(1,0){16}}
\put(72,20){\line(-1,1){8}}
\put(72,8){\line(-1,-1){8}}
\put(72,20){\line(0,-1){12}}
\end{picture}
\caption{The unique 3-regular graph with largest least eigenvalue less than $-2$.}\label{largestleast}
\end{figure}
with smallest eigenvalue $\lambda_{\min}=\gamma\approx-2.0391$, where $\gamma$ is the smallest root of $f(x)=x^6-3x^5-7x^4+21x^3+13x^2-35x-4$, and moreover proved that there is no connected, 3-regular graph with smallest eigenvalue in the interval $(\gamma,-2)$ (that is, among all connected, 3-regular graphs $G$ has the largest least eigenvalue less than $-2$). Since the second eigenvalue of the complement of a regular graph is $\lambda_2=-1-\lambda_{\min}$, the complement $\overline{G}$ of $G$, a 12-regular graph on 16 vertices, has second eigenvalue $\lambda_2(\overline{G})=-1-\gamma\approx1.0391$. We do not know if $\overline G$ has smallest second eigenvalue larger than 1 among regular graphs, but it is not unique. Indeed, 
the complement of the disjoint union $G+kK_4$ of $G$ and $k$ copies of $K_4$ is a connected, $(12+4k)$-regular graph on $16+4k$ vertices with second eigenvalue $\lambda_2(\overline{G+kK_4})=-1-\gamma$, so we have found an infinite family of regular graphs with second eigenvalue $-1-\gamma$. 

\begin{problem}
For any integer $k\geq 2$, let $\lambda(k) := (-1+ \sqrt{4k-3})/2$. Then we find that $v(k, \lambda(k))\leq k^2 + 1$ with equality if and only if the associated graph is a Moore graph of diameter $2$.  
Moore graphs of diameter $2$ only exists for $k=2, 3, 7$, and possibly $57$. If $k$ is not $2,3,7,57$, then $v(k, \lambda(k)) \leq k^2$. Determine the exact value of $v(k, \lambda(k))$ in these cases.
\end{problem}

An $(n,k,\lambda)$-graph is a $k$-regular graph with $n$ vertices such that $|\lambda_i|\leq \lambda$ for $i\geq 2$. This notion was introduced by Alon (see \cite{A,KS}) motivated by the study of pseudo-random graphs and expanders among other things. The following question seems natural and interesting. 
\begin{problem}
Given $k\geq 3$ and $1<\lambda<2\sqrt{k-1}$, what is the maximum order $n$ of an $(n,k,\lambda)$-graph ?
\end{problem}

\section*{Acknowledgments}

The authors thank Joel Friedman, Chris Godsil, Bill Martin and an anonymous referee for some useful comments and suggestions.

\end{document}